\documentclass[11pt]{article}  

\usepackage{amsmath, amsthm, amsfonts, amssymb}
\usepackage[bookmarks]{hyperref}
\usepackage{indentfirst} 
\usepackage{marvosym}

\usepackage[a4paper]{geometry}

\newtheorem{lemma}{Lemma}[section]
\newtheorem{theorem}[lemma]{Theorem}
\newtheorem{proposition}[lemma]{Proposition}
\newtheorem{corollary}[lemma]{Corollary}
\renewenvironment{proof}[1][\proofname]{{\sc #1. }}{\qed}
\theoremstyle{definition}
\newtheorem{remark}[lemma]{Remark}
\newtheorem{example}[lemma]{Example}
\newtheorem{defin}[lemma]{Definition}{\bf}{\rm}

\newcommand{\abs}[1]{\ensuremath{\left| #1 \right|}}
\newcommand{\op}{\operatorname}
\newcommand{\ce}[2]{\operatorname{C}_{#1}(#2)}
\newcommand{\no}[2]{\operatorname{N}_{#1}(#2)}
\newcommand{\ze}[1]{\operatorname{Z}(#1)}
\newcommand{\fra}[1]{\op{\Phi}(#1)}
\newcommand{\fit}[1]{\operatorname{F}(#1)}
\newcommand{\rad}[2]{\op{O}_{#1}(#2)}
\newcommand{\syl}[2]{\op{Syl}_{#1}\left(#2\right)}
\newcommand{\hall}[2]{\op{Hall}_{#1}\left(#2\right)}

\begin{document}

\title{\bf Zeros of irreducible characters in factorised groups}

\author{\sc M. J. Felipe $\cdot$ A. Mart\'inez-Pastor $\cdot$ V. M. Ortiz-Sotomayor
\thanks{The first author is supported by Proyecto Prometeo II/2015/011, Generalitat Valenciana (Spain). The second author is supported by Proyecto MTM 2014-54707-C3-1-P, Ministerio de Econom\'ia, Industria y Competitividad (Spain), and by Proyecto Prometeo/2017/057, Generalitat Valenciana (Spain). The third author acknowledges the predoctoral grant ACIF/2016/170, Generalitat Valenciana (Spain).
\rule{6cm}{0.1mm}\newline
Instituto Universitario de Matemática Pura y Aplicada (IUMPA-UPV), Universitat Polit\`ecnica de Val\`encia, Camino de Vera s/n, 46022 Valencia, Spain\newline
\Letter: \texttt{mfelipe@mat.upv.es}, \texttt{anamarti@mat.upv.es}, \texttt{vicorso@doctor.upv.es} \newline
ORCID iDs: 0000-0002-6699-3135, 0000-0002-0208-4098, 0000-0001-8649-5742
}}

\date{}

\maketitle

\begin{abstract}
\noindent An element $g$ of a finite group $G$ is said to be \emph{vanishing in} $G$ if there exists an irreducible character $\chi$ of $G$ such that $\chi(g)=0$; in this case, $g$ is also called a \emph{zero} of $G$. The aim of this paper is to obtain structural properties of a factorised group $G=AB$ when we impose some conditions on prime power order elements $g\in A\cup B$ which are (non-)vanishing in $G$. 

\medskip


\smallskip

\end{abstract}


\section{Introduction}

Within finite group theory, the close relationship between character theory and the study of conjugacy classes is widely known. Regarding this last topic, several authors have investigated the connection between certain conjugacy class sizes (also called indices of elements) of a group $G$ and its structure. Further, recent results show up that the conjugacy classes of the elements in the factors of a factorised group exert a strong impact on the structure of the whole group (see \cite{BCL, CL, FMOsquare, FMOprime}). 

In character theory, a celebrated Burnside's result asserts: every row in a character table of a finite group which corresponds to a non-linear complex character has a zero entry \cite[Theorem 3.15]{I}. Nevertheless, a non-central conjugacy class column may not contain a zero. This fact somehow violates the standard duality arising in many cases between the two referred research lines. Therefore, in \cite{INW} the authors introduce the next concept: an element $g\in G$ is \emph{vanishing in} $G$ if there exists an irreducible character $\chi$ of $G$ such that $\chi(g)=0$ (in the literature, $g$ is also called a \emph{zero} of $\chi$). Otherwise the element $g$ is said to be \emph{non-vanishing in} $G$. As an immediate consequence of the cited Burnside's result, we get that a group has no vanishing elements if and only if it is abelian. It is to be said that various questions concerning (non-)vanishing elements have been studied by numerous authors (in particular, those appearing as references in this paper).

It is therefore natural to wonder whether results based on conjugacy class sizes remain true if we restrict focus only to those indices that correspond to vanishing elements, i.e. if we consider only the \textit{vanishing indices}. In this spirit, some researchers have recently obtained positive results in certain cases. For instance, in 2010, Dolfi, Pacifici and Sanus proved that if a prime $p$ does not divide each vanishing index of a group $G$, then $G$ has a normal $p$-complement and abelian Sylow $p$-subgroups \cite[Theorem A]{DPS}. In 2016, Brough showed that for a fixed prime $p$ such that $(p-1, \abs{G})=1$, if all vanishing indices of $G$ are not divisible by $p^2$, then $G$ is soluble \cite[Theorem A]{B}. Moreover, if each vanishing index of $G$ is square-free, then $G$ is supersoluble \cite[Theorem B]{B}. The last two results turn to be the ``vanishing versions'' of \cite[Theorem 1]{CW} and \cite[Theorem 2]{CW}, respectively. Besides, Brough and Kong have also showed in \cite{BK} that the hypotheses in the previous results can be weakened to vanishing indices of prime power order elements. We remark that the classification of finite simple groups (CFSG) is used in this development.

In this paper, we are interested in combining as a novelty the research on irreducible characters with the study of products of groups. More concretely, we want to analyse which information of a factorised group $G$ can be obtained from its character table when we consider the conjugacy classes in $G$ of elements in the factors. In particular, inspired by the aforementioned investigations, we deal with factorised groups having irreducible characters which evaluate zero on some elements in the factors. It is worthwhile to note that the product of two vanishing elements needs not to be vanishing in general. Moreover, an element in a (normal) subgroup can be vanishing in the whole group but not in that subgroup (see Example \ref{example_van}).

Focusing in products of groups, along the last decades, some relations of permutability between the factors have been considered by many authors, as for instance total permutability, mutual permutability (see \cite{BEA}) and tcc-permutability (see \cite{AAP, AAMP}). These last permutability relations are inherited by quotients, and they ensure the existence of a minimal normal subgroup contained in one of the factors. We are principally concerned about products of groups that satisfy both particular conditions, which we will name \emph{core-factorisations} (see Definition \ref{definition_core}). 

In this framework, our purpose is to get a better understanding of how the vanishing elements in the factors control the structure of a group with a core-factorisation. Moreover, we will also deal with arithmetical conditions on the indices of those elements.

The paper is structured in the following way: Firstly, core-factorisations are defined in Section \ref{sec_core} and some properties of them, which will be crucial along the paper, are proved. In Section \ref{sec_van}, we analyse the case that a group with a core-factorisation has no vanishing $p$-elements in the factors for a prime $p$ (see Theorem \ref{SylowNormal}). As a consequence, we obtain information of a factorised group when all prime divisors of its order are considered, that is, when there are no vanishing prime power order elements in the factors (see Corollary \ref{propositionKEY}). Later on we obtain structural properties of groups with a core-factorisation from the vanishing indices in the whole group of some elements in the factors. Concretely, in Section \ref{sec_prime}, we study the case when those vanishing indices are prime powers (Theorem \ref{theoremCCvanishing} and Corollary \ref{G/Fabelian}). Next, we focus in Section \ref{sec_square} on the case that the indices are not divisible by a prime $p$ (see Theorem \ref{theorem_p-reg}). The situation when those indices are square-free is also handled in this last section (see  Theorems \ref{theoremp2} and \ref{theoremSquare}). In particular, we highlight that an affirmative answer to a question posed by Brough in \cite{B} is given (Corollary \ref{corollary_pnilp}). It is significant to mention again that all the previous results for core-factorisations will remain true when the factors are either totally, mutually or tcc-permutable (see Example \ref{exampleCore}). We remark that, in order to avoid repeating arguments from previous papers, when some proof runs as in the one of a known result with suitable changes, we refer to the corresponding one.

Throughout this paper, every group is assumed to be finite. The terminology here is as follows: for a group $G$ and an element $x\in G$, we call $i_G(x)$ the \emph{index} of $x$ in $G$, that is, $i_G(x)=\abs{G:\ce{G}{x}}$ is the size of the conjugacy class $x^G$. The set of prime divisors of the order of $G$ is denoted by $\pi(G)$. If $p$ is a prime, then $x\in G$ is a $p$\emph{-regular} element if its order is not divisible by $p$. As customary, the set of all Sylow $p$-subgroups of $G$ is denoted by $\syl{p}{G}$, whilst $\hall{\pi}{G}$ is the set of all Hall $\pi$-subgroups of $G$ for a set of primes $\pi$. We write $\op{Irr}(G)$ for the set of all irreducible complex characters of $G$. Given a group $G=AB$ which is the product of the subgroups $A$ and $B$, a subgroup $S$ is called \emph{prefactorised} (with respect to this factorisation) if $S=(S\cap A)(S\cap B)$ (see \cite{BEA}). We recall that a subgroup $U$ \emph{covers} a section $V/W$ of a group $G$ if $W(U\cap V)=V$. The remainder notation is standard, and it is taken mainly from \cite{DH}. In particular, a normal subgroup $N$ of a group $G$ such that $N\neq G$ is denoted symbolically by $N\lhd G$. We also refer to \cite{DH} for details about classes of groups.


\section{Core-factorisations: definition and properties}
\label{sec_core}

We analyse in this section the kind of factorisations we manage along the paper.

\begin{defin}
Let $1\neq G=AB$ be the product of the subgroups $A$ and $B$. We say that $G=AB$ is a \emph{core-factorisation} if for every proper normal subgroup $K$ of $G$ it holds that there exists a normal subgroup $1\neq M/K$ of $G/K$ such that either $M/K\leqslant AK/K$ or $M/K\leqslant BK/K$ (i.e. either $A$ or $B$ covers $M/K$).
\label{definition_core}
\end{defin}

Note that if we adopt the bar convention for the quotients over $K$, the above condition means that $\overline{A}_{\overline{G}} \overline{B}_{\overline{G}}\neq 1$, where $H_X$ denotes the core in a group $X$ of a subgroup $H$. This illustrates the given name for such factorisations.

\begin{remark}
Let state some immediate facts:
\begin{enumerate}
	\item If either $1\neq G=A$ or $1\neq G=B$, then $G=AB$ is always a core-factorisation. 
	
	\item If $G=AB$ is a core-factorisation of a simple group $G$, then either $G=A$ or $G=B$.
	
	\item If we take $K=1$ in the above definition, then there exists a (minimal) normal subgroup of $G=AB$ contained in either $A$ or $B$.
\end{enumerate}
\end{remark}

We present now some non-trivial examples.

\begin{example}
\label{exampleCore}
Let $1\neq G=AB$ be the product of the subgroups $A$ and $B$, and let assume that $A$ and $B$ satisfy one of the following permutability properties: 
\begin{itemize}
	\item[(i)] $A$ and $B$ are mutually permutable, that is, $A$ permutes with every subgroup of $B$ and $B$ permutes with every subgroup of $A$.

	\item[(ii)] $A$ and $B$ are tcc-permutable, that is, if for every subgroup $X$ of $A$ and every subgroup $Y$ of $B$, there exists $g\in\langle X, Y\rangle$ such that $X$ permutes with $Y^g$.
	
	\item[(iii)] $A$ and $B$ are totally permutable, that is, every subgroup of $A$ permutes with every subgroup of $B$. (In particular, if this property holds, then $A$ and $B$ satisfy both (i) and (ii).)
\end{itemize}
Applying \cite[Theorem 4.3.11]{BEA} in (i) and \cite[Lemma 2.5]{AAP} in (ii), it can be seen that $A_G B_G\neq 1$. Also, the above permutability properties are clearly inherited by quotients. Thus $G=AB$ is a core-factorisation in all cases. We shall see later in Example \ref{tcc-core} a group with with a core-factorisation whose factors are neither mutually permutable nor tcc-permutable.
\end{example}

Now we prove that the quotients of core-factorisations inherit the property.

\begin{lemma}
\label{lemacore}
Let $G=AB$ be a core-factorisation, and let $M$ be a proper normal subgroup of $G$. Then $G/M=(AM/M)(BM/M)$ is also a core-factorisation.
\end{lemma}

\begin{proof}
Let use the bar convention to denote the quotients over $M$. We take a normal subgroup $\overline{K} \lhd \overline{G}$, and we claim that there exists a normal subgroup $1\neq \overline{N}/\overline{K}$ of $\overline{G}/\overline{K}$ covered by either $\overline{A}$ or $\overline{B}$. As $G = AB$ is a core-factorisation, then $G/K$ has a normal subgroup $1\neq N/K$ such that either $N/K$ is covered by either $A$ or $B$.  It follows $$1\neq \frac{\overline{N}}{\overline{K}} = \frac{N/M}{K/M} \leqslant \frac{AK/M}{K/M}=\frac{(AM/M)(K/M)}{K/M}=\frac{\overline{A}\overline{K}}{\overline{K}},$$ or analogously the same is valid for $B$ instead of $A$.
\end{proof}

\medskip

The lemma below is a characterisation of core-factorisations via normal series.

\begin{lemma}
\label{core_charac}
Let $1\neq G=AB$ be the product of the subgroups $A$ and $B$. The following statements are pairwise equivalent:
\begin{enumerate}
	\item[(i)] $G=AB$ is a core-factorisation.
	
	\item[(ii)] There exists a normal series $1=N_0 \unlhd N_1 \unlhd \cdots \unlhd N_{n-1} \unlhd N_n=G$ such that either $N_i/N_{i-1} \leqslant AN_{i-1}/N_{i-1}$ or $N_i/N_{i-1} \leqslant BN_{i-1}/N_{i-1}$, for each $1\leq i \leq n$ (i.e. $N_i/N_{i-1}$ is covered by either $A$ or $B$).
	
	\item[(iii)] There exists a chief series $1=N_0 \unlhd N_1 \unlhd \cdots \unlhd N_{n-1} \unlhd N_n=G$ such that either $N_i/N_{i-1} \leqslant AN_{i-1}/N_{i-1}$ or $N_i/N_{i-1} \leqslant BN_{i-1}/N_{i-1}$, for each $1\leq i \leq n$ (i.e. $N_i/N_{i-1}$ is covered by either $A$ or $B$).
\end{enumerate}  
Further, each term $N_i$ of such (chief) normal series is prefactorised and $N_i=(N_i\cap A)(N_i\cap B)$ is also a core-factorisation.
\end{lemma}

\begin{proof}
(i) implies (ii): Let $1\neq N_1\unlhd G$ such that either $N_1\leqslant A$ or $N_1\leqslant B$, so $1\lhd N_1\unlhd G$. Next, take $G/N_1=(AN_1/N_1)(BN_1/N_1)$. If $G/N_1=1$, then we have the desired series. If $1\neq G/N_1$, then it is again a core-factorisation by the previous lemma. Therefore, there exists $1\neq N_2/N_1\unlhd G/N_1$ such that either $N_2/N_1\leqslant AN_1/N_1$ or $N_2/N_1\leqslant BN_1/N_1$. So we get the series $1\lhd N_1 \lhd N_2\unlhd G$. Repeating this process until we reach a trivial quotient $G/N_{j}$, we get the desired series.

(ii) implies (iii):  If we refine the series in (ii) to a chief series, then we get for each factor that there exist $N_i=T_0 \unlhd T_1 \unlhd T_2 \unlhd \cdots \unlhd T_k= N_{i+1}$ such that each $T_j/T_{j-1}$ is a minimal normal subgroup of $G/T_{j-1}$. Let see that either $T_j/T_{j-1}\leqslant AT_{j-1}/T_{j-1}$ or $T_j/T_{j-1}\leqslant BT_{j-1}/T_{j-1}$. We may assume for instance $N_{i+1}\leqslant AN_i$. Thus $T_j=T_j\cap N_{i+1}\leqslant N_i(T_j\cap A)\leqslant T_{j-1}A$, and so $T_j/T_{j-1}\leqslant AT_{j-1}/T_{j-1}$.

(iii) implies (i): We have to show that for each $K\lhd G$, there exists a non-trivial normal subgroup of $G/K$ covered by either $A$ or $B$. Let $1\leq r \leq n$ be the minimum number such that $N_r\nleqslant K$. Then $1\neq N_rK/K$ is normal in $G/K$. Let suppose for instance that $N_r/N_{r-1}\leqslant AN_{r-1}/N_{r-1}$, so $N_r \leqslant AN_{r-1}$. By the minimality of $r$ it follows $N_rK/K \leqslant AK/K$.

Now we claim that each $N_i$ in such (chief) normal series is prefactorised, and we work by induction on $i$. The case $i=1$ is clear since either $N_1\leqslant A$ or $N_1\leqslant B$. Now we assume that $N_{i-1}=(N_{i-1} \cap A)(N_{i-1}\cap B)$ and we want to show that $N_i$ is also prefactorised. We may consider $N_i\leqslant AN_{i-1}$, and then $N_i=(N_i\cap A)N_{i-1}=(N_i\cap A)(N_{i-1} \cap A)(N_{i-1}\cap B)\subseteq (N_i\cap A)(N_i\cap B)\subseteq N_i$.

Fix a prefactorised $N_i=(N_i\cap A)(N_i\cap B)$ of a (chief) normal series of $G$ like in (ii) or (iii), for some $i\in\{1,\ldots, n\}$. We are showing that $N_i=(N_i\cap A)(N_i\cap B)$ is a core-factorisation. Consider the following portion of such (chief) normal series $1=N_0 \unlhd N_1 \unlhd \cdots \unlhd N_{i}$. Let $m\in\{1, \ldots, i\}$. We claim that $N_m$ satisfies either $N_m\leqslant (N_i\cap A)N_{m-1}$ or $N_m\leqslant (N_i\cap B)N_{m-1}$ in order to apply the equivalence between (ii) and (i). We have by assumption that for instance $N_m\leqslant N_{m-1}A$, so $N_m\leqslant N_{m-1}A \cap N_i=N_{m-1}(A\cap N_i)$. The lemma is now established. 
\end{proof}

\medskip

We point out that if $N$ is an arbitrary prefactorised normal subgroup of a core-factorisation $G=AB$, then $N=(N\cap A)(N\cap B)$ might not be a core-factorisation, as the next example shows.

\begin{example}
Consider $G=\op{Sym}(4) \times \langle x\rangle$, where $\op{Sym}(4)$ denotes the symmetric group of $4$ letters and $o(x)=2$. If $A=\langle ((1,2), x),\: ((3,4), x), \: ((1,3)(2,4), x)\rangle$ and $B=\langle ((2,3,4), 1), \: ((3,4), 1), \:(1, x)\rangle$, then $G=AB$ is a core-factorisation, and $N=\op{Sym}(4)=(N\cap A)(N\cap B)$ is not a core-factorisation, since there is no minimal normal subgroup of $N$ neither in $N\cap A$ nor in $N\cap B$. Moreover, it can be seen that $A$ and $B$ are not either mutually nor tcc-permutable.
\label{tcc-core}
\end{example}


\section{On vanishing elements}
\label{sec_van}

The main objective of this section is to prove Theorem \ref{SylowNormal} and Corollary \ref{propositionKEY}. Let state first some key ingredients for locating vanishing elements in a given group.

\begin{lemma}\emph{\cite[Lemma 2.9]{DPSS}}
\label{contradiction_lemma}
Let $N\leqslant M\leqslant G$, with $N$ and $M$ normal in $G$ and $(\abs{N}, \abs{M/N})=1$. If $N$ is minimal normal in $G$, $\ce{M}{N}\leqslant N$ and $M/N$ is abelian, then every element in $M\smallsetminus N$ is vanishing in $G$.
\end{lemma}

In 2017, Bianchi, Brough, Camina and Pacifici obtained the subsequent result.

\begin{lemma}\emph{\cite[Corollary 4.4]{BBCP}}
\label{contradiction_lemma2}
Let $G$ be a group, and $K$ an abelian minimal normal subgroup of $G$. Let $M/N$ be a chief factor of $G$ such that $(\abs{K}, \abs{M/N})=1$, and $N = \ce{M}{K}$. Then every element of $M\smallsetminus N$ is a vanishing element of $G$.
\end{lemma}

Let $p$ be a prime, and $\chi\in\op{Irr}(G)$. Recall that $\chi$ is of $p$\emph{-defect zero} if $p$ does not divide $\frac{\abs{G}}{\chi(1)}$. A well-known result of Brauer \cite[Theorem 8.17]{I} highlights the significance that this property has for vanishing elements: if $\chi$ is an irreducible character of $p$-defect zero of $G$ then, for every $g\in G$ such that $p$ divides the order of $g$, it holds $\chi(g)=0$. The following lemma yields elements of normal subgroups that vanish in the whole group.

\begin{lemma}\emph{\cite[Lemma 2.2]{B}}
\label{Normaldefect}
Let $N$ be a normal subgroup of a group $G$. If $N$ has an irreducible character of $p$-defect zero, then every element of $N$ of order divisible by $p$ is a vanishing element in $G$.
\end{lemma}

We now focus on vanishing elements in simple groups. The combination of some results in \cite{DPSS}, which use the classification, gives the following.

\begin{proposition}
\label{simplegroups}
Let $S$ be a non-abelian simple group, and let $p\in\pi(S)$. Then, either there exists $\chi\in\op{Irr}(S)$ such that $\chi$ is of $p$-defect zero, or there exists a $p$-element $x\in S$ and $\chi\in\op{Irr}(S)$ such that $\chi$ extends to $\op{Aut}(S)$ and $\chi$ vanishes on $x$.
\end{proposition}

\begin{proof}
If either $S$ is a group of Lie type or $p\geq 5$, then \cite[Proposition 2.1]{DPSS} applies and $S$ has an irreducible character of $p$-defect zero (note that this case includes the groups $A_5 \cong PSL(2, 5)$ and $A_6 \cong PSL(2, 9)$). Hence it remains to consider sporadic simple groups and alternating groups, and $p\in \{2, 3\}$. Firstly, in virtue of \cite[Lemma 2.3]{DPSS}, for a sporadic simple group $S$ there exists always an irreducible character which extends to $\op{Aut}(S)$ and it vanishes on a $p$-element. For alternating groups $A_n$ with $n \geq  7$, it is known by \cite[Proposition 2.4]{DPSS} that $A_n$ has two irreducible characters $\chi_2, \chi_3$ such that $\chi_2$ vanishes on a $2$-element and $\chi_3$ vanishes on an element of order $3$. Further, both $\chi_2$ and $\chi_3$ extend to $\op{Aut}(A_n)$. 
\end{proof}

\medskip

An argument included within the proof of \cite[Theorem A]{DPSS} provides the following proposition, which turns to be essential in the remainder of the section.

\begin{proposition}
\label{reduction_lemma}
Let $N$ be a non-abelian minimal normal subgroup of a finite group $G$, and let $p\in\pi(N)$. Then there exists a $p$-element in $N$ which is vanishing in $G$.
\end{proposition}

\begin{proof}
We have that $N=S_1\times \cdots \times S_k$, where each $S_i$ is isomorphic to a non-abelian simple group $S$ with $p$ dividing its order. If $S$ has a character $\theta$ of $p$-defect zero, then $\chi := \theta \times \cdots \times \theta \in \op{Irr}(N)$ and it is clear that $\chi$ is also of $p$-defect zero. Let $1\neq x_i\in S_i$ be a $p$-element. Then $1\neq x:=x_1\cdots x_k \in N$ is a $p$-element and Lemma \ref{Normaldefect} provides that $x$ is vanishing in $G$. 

Let $i\in\{1, \ldots, k\}$ and suppose that $S_i$ does not have a character of $p$-defect zero. By Proposition \ref{simplegroups}, there exists $\theta \in \op{Irr}(S_i)$ and a $p$-element $y_i\in S_i$ such that $\theta(y_i)=0$ (so $1\neq y_i$) and $\theta$ extends to $\op{Aut}(S_i)$. Thus $1\neq y:=y_1\cdots y_k\in N$ is a $p$-element, and by \cite[Proposition 2.2]{DPSS} it follows that $\chi:=\theta\times\cdots\times\theta\in\op{Irr}(N)$ extends to $G$. Moreover, $\chi(y)=0$, and the result is now established.
\end{proof}

\medskip

From now on we deal with (non-)vanishing elements in factorised groups. The next example gives insight into occurring phenomena.

\begin{example}
\label{example_van}
Let $G=\op{Sym}(4)\times \langle x \rangle = AB$ be the factorised group as in Example \ref{tcc-core}. Note that although $((3, 4), x)$ is vanishing in $A$ and $((3, 4), 1)$ is vanishing in $B$, the product $((3, 4), x)((3, 4), 1) = (1, x)\in\ze{G}$ and so it is non-vanishing in $G$. On the other hand, $((2, 3, 4), 1)$ is a non-vanishing element in $B$ which is vanishing in $G$.
\end{example}

\begin{remark}
\label{remarkhyp}
We claim that the hypotheses regarding vanishing elements of the results stated from now on are inherited by every non-trivial quotient of a group $G$, where $G=AB$ is a core-factorisation. Indeed, let $N$ be a proper normal subgroup of $G$. Note that $G/N=(AN/N)(BN/N)$ is also a core-factorisation by Lemma \ref{lemacore}. Since there exists a bijection between $\op{Irr}(G)$ and the  set of all characters in $\op{Irr}(G/N)$ containing $N$ in their kernel, if $xN\in AN/N\cup BN/N$ is a vanishing (prime power order) element of $G/N$, then we can assume $x\in A\cup B$, and that $x$ is also a vanishing (prime power order) element of $G$. This fact will be used in the sequel, sometimes with no reference.
\end{remark}

Our first significant result analyses core-factorisations with no vanishing $p$-elements in the factors. We remark that the CFSG is needed.

\begin{theorem}
\label{SylowNormal}
Let $G=AB$ be a core-factorisation, and let $p$ be a prime. If every $p$-element in $A\cup B$ is non-vanishing in $G$, then $G$ has a normal Sylow $p$-subgroup.
\end{theorem}

\begin{proof}
Let $G$ be a counterexample of minimal order to the result, and take $P\in\syl{p}{G}$. Clearly we can assume that $\rad{p}{G}$ is proper in $G$. Hence by Remark \ref{remarkhyp} and the minimality of $G$ we may suppose $\rad{p}{G}=1$. Since $G=AB$ is a core-factorisation, we can consider a minimal normal subgroup $N$ of $G$ such that $N\leqslant A$, for instance. Let suppose that $p$ divides its order. Then $N$ is non-abelian, and by Proposition \ref{reduction_lemma} there is a $p$-element $x\in N$ which is vanishing in $G$, a contradiction. So $p$ does not divide the order of $N$. In particular, we may assume that $N$ is proper in $G$. By minimality and Remark \ref{remarkhyp} we obtain that $PN/N$ is normal in $G/N$, and then $G$ is $p$-separable. 

We can choose by Lemma \ref{core_charac} a chief series $1=N_0 \unlhd N_1=N \unlhd \cdots \unlhd N_{n-1} \unlhd N_n=G$ such that each chief factor $N_i/N_{i-1}$ is covered by either $A$ or $B$. Let $j\in\{2, \ldots, n\}$ be the minimum number such that $p$ divides $\abs{N_j/N_{j-1}}$. Then $N_j/N_{j-1}$ is a minimal normal subgroup of $G/N_{j-1}$ and it is $p$-elementary abelian. It follows that $N_j/N = N_{j-1}/N \times P_0/N$, where $1\neq P_0/N=PN/N\cap N_j/N$ is the unique Sylow $p$-subgroup (and elementary abelian) of $N_j/N$. We claim that every element of $P_0\smallsetminus N$ is vanishing in $G$. Note that $P_0/N$ is abelian and normal in $G/N$. It also holds $(\abs{N}, \abs{P_0/N})=1$. In addition, since $N=\rad{p'}{P_0}$ and $\rad{p}{P_0}\leqslant\rad{p}{G}=1$, then $\ce{P_0}{N}\leqslant N$. Lemma \ref{contradiction_lemma} yields that every element in $P_0\smallsetminus N$ is vanishing in $G$. Therefore, it remains to find a $p$-element in $P_0\smallsetminus N$ lying in either $A$ or $B$ in order to get the final contradiction. 

Since $N_j=(N_j\cap A)(N_j \cap B)$ by Lemma \ref{core_charac}, applying \cite[Lemma 2]{FMOprime} we can affirm that the unique Sylow $p$-subgroup $P_0/N$ of $N_j/N$ is also prefactorised, that is, $P_0/N=(P_0/N \cap (N_j\cap A)/N)(P_0/N \cap (N_j\cap B)N/N)$. Let $X\in \{A, B\}$ such that $1\neq P_0/N\cap (N_j\cap X)N/N=(P_0 \cap N_j \cap X)N/N$. If we pick a $p$-element $1\neq x \in (P_0 \cap N_j \cap X) \smallsetminus N$, then $x$ is vanishing in $G$. Hence the result is established.
\end{proof}

\medskip

As an immediate consequence, when we take the trivial factorisation $G=A=B$ in the above theorem, we obtain \cite[Theorem A]{DPSS}. In their proof, the authors apply Lemma \ref{contradiction_lemma} to the centre of a Sylow subgroup in order to get the final contradiction. We highlight that the centre subgroup may not be prefactorised (see \cite[Example 4.1.43]{BEA}) and so our reasonings differ.

Another consequence of Theorem \ref{SylowNormal} is the following.

\begin{corollary}
\label{corollaryHall}
Let $G=AB$ be a core-factorisation, and let $\sigma$ be a set of primes. If every $\sigma$-element of prime power order in $A\cup B$ is non-vanishing in $G$, then $G$ has a nilpotent normal Hall $\sigma$-subgroup.
\end{corollary}

\begin{proof}
Apply Theorem \ref{SylowNormal} for each prime in $\sigma$.
\end{proof}

\medskip

Note that if $\sigma=p'$ in the above result, then it generalises \cite[Corollary B]{DPSS}. Indeed, the next corollary extends \cite[Corollary C]{DPSS} for factorised groups.

\begin{corollary}
Let $G=AB$ be a core-factorisation, and let $\{ p, q\}\subseteq\pi(G)$. If every element in $A\cup B$ vanishing in $G$ has order a $\{p, q\}$-number, then $G$ is soluble.
\end{corollary}

\begin{proof}
We denote by $\sigma:=\{p, q\}'$. In virtue of Corollary \ref{corollaryHall}, $G$ has a nilpotent normal Hall $\sigma$-subgroup $N$. Now, $G/N$ is soluble because it is a $\{p, q\}$-group, so $G$ is also soluble. 
\end{proof}

\medskip

If we consider the case when the hypotheses in Theorem \ref{SylowNormal} hold for all primes, then it follows clearly that those groups are nilpotent. But actually we obtain the stronger fact that they are abelian. The next result is essential in its proof.

\begin{proposition}\emph{\cite[Theorem B]{INW}}
\label{INW_supersoluble}
If $G$ is supersoluble, then every element in $G\smallsetminus \ze{\fit{G}}$ is vanishing in $G$. In particular, if $G$ is nilpotent, then all elements in $G\smallsetminus \ze{G}$ are vanishing in $G$.
\end{proposition}

\begin{corollary}
\label{propositionKEY}
Let $G=AB$ be a core-factorisation. The following statements are pairwise equivalent:
\begin{enumerate}
	\item[\emph{(1)}] Every element $x\in A\cup B$ is non-vanishing in $G$.
	
	\item[\emph{(2)}] Every prime power order element $x\in A\cup B$ is non-vanishing in $G$.
	
	\item[\emph{(3)}] $G$ is abelian.
\end{enumerate}
\end{corollary}

\begin{proof}
There is no doubt in the implications (1) $\Rightarrow$ (2) and (3) $\Rightarrow$ (1), so let prove (2) $\Rightarrow$ (3). Clearly, by Theorem \ref{SylowNormal}, $G$ is nilpotent. Since we are assuming that every prime power order element lying in $A\cup B$ is non-vanishing in $G$, then Proposition \ref{INW_supersoluble} provides that every Sylow subgroup of $A$ and $B$ lies below $\ze{G}$, and thus $G=AB \leqslant \ze{G}$. 
\end{proof}

\medskip

As it has been said before, from Burnside's result quoted in the introduction it is elementary to show that a group is abelian if and only if it has no vanishing elements. Indeed, it is enough to consider in this last characterisation only prime power order elements, as we directly deduce by taking the trivial factorisation in the previous corollary. This claim can be also obtained from \cite[Theorem B]{MNO}, which asserts that a non-linear complex character vanishes on a prime power order element (it also uses the CFSG). In any case, both proofs emphasize the difficulty of handling only prime power order elements. Moreover, observe that \cite[Theorem B]{MNO} does not imply directly Corollary \ref{propositionKEY}, since we cannot assure in a factorised group that a vanishing prime power order element lies in one of the factors.


\section{Prime power vanishing indices}
\label{sec_prime}

In \cite{CC}, Camina and Camina analysed the structure of the so-called $p$-Baer groups, i.e. groups all of whose $p$-elements have prime power indices for a given prime $p$. Next, in \cite{FMOprime} we extended this study through products of two arbitrary groups. Thus, as stated in the introduction, it seems natural to address the corresponding vanishing problem, i.e. vanishing indices which are prime powers, in particular for factorised groups.

Let enunciate first some preliminary results. The subsequent well-established one is due to Wielandt.

\begin{lemma}
\label{wielandt}
Let $G$ be a finite group and $p$ a prime. If $x\in G$ is a $p$-element and $i_G(x)$ is a $p$-number, then $x\in\op{O}_{p}(G)$.
\end{lemma}

In \cite{CC}, Camina and Camina proved the next proposition, which extends both the above lemma and the celebrated Burnside's result about the non-simplicity of groups with a conjugacy class of prime power size.

\begin{proposition}\emph{\cite[Theorem 1]{CC}}
\label{CaminaCamina}
All elements of prime power index of a finite group $G$ lie in $\op{F}_2(G)$, the second term of the Fitting series of $G$.
\end{proposition}

The main result of \cite{B_non} is the following one.

\begin{proposition}
\label{lemma_non_brough}
Let $G$ be a group which contains a non-trivial normal $p$-subgroup $N$ for $p$ a prime. Then each $x\in N$ such that $p$ does not divide $i_G(x)$ is non-vanishing in $G$.
\end{proposition}

Finally, the lemma below is elementary.

\begin{lemma}
Let $N$ be a normal subgroup of a group $G$, and $A$ be a subgroup of $G$. Let $p$ be a prime. We have:
\begin{itemize}
	\item[(a)] $i_N(x)$ divides $i_G(x)$, for any $x\in N$.
	
	\item[(b)] $i_{G/N}(xN)$ divides $i_G(x)$, for any $x\in G$.
\end{itemize}
\end{lemma}

\begin{remark}
Note that, hereafter, in the results stated the arithmetical hypotheses on the indices are inherited by non-trivial quotients of core-factorisations. Indeed, let $G=AB$ be a core-factorisation and suppose for an element $x\in A\cup B$ that $i_G(x)$ is a prime power, square-free, or not divisible by a given prime, respectively. Since $i_{G/N}(xN)$ divides $i_G(x)$ by the above lemma, we get that $i_{G/N}(xN)$ is also a prime power, square-free, or not divisible by such prime, respectively.
\end{remark}

We are now ready to prove the following vanishing versions of \cite[Theorem A (1-2)]{FMOprime} and \cite[Theorem B (1)]{FMOprime} for core-factorisations, respectively. We emphasize that the techniques used in that approach are not valid when we work with zeros of irreducible characters.

\begin{theorem}
\label{theoremCCvanishing}
Let $G=AB$ be a core-factorisation. Let $p$ be a prime, and $P\in\syl{p}{G}$. Assume that every $p$-element $x\in A\cup B$ vanishing in $G$ has prime power index. Then:
\begin{enumerate}
	\item[\emph{(1)}] If all the considered indices are $p$-numbers, then $P$ is normal in $G$.

	\item[\emph{(2)}] $G/\ce{G}{\rad{p}{G}}$ has a normal Sylow $p$-subgroup.

	\item[\emph{(3)}] $G/\fit{G}$ has a normal Sylow $p$-subgroup. 
	
	\item[\emph{(4)}] $G/\rad{p'}{G}$ has a normal Sylow $p$-subgroup. So $G$ is $p$-soluble of $p$-length 1.
\end{enumerate}
\end{theorem}

\begin{proof}
(1) If all the indices of vanishing $p$-elements $x\in A\cup B$ are $p$-numbers, then it is enough to reproduce the proof of Theorem \ref{SylowNormal}. Notice that the contradictions now will be derived from Lemma \ref{wielandt}.

(2) Let denote $\overline{G}:=G/\ce{G}{\rad{p}{G}}$. We may assume $\overline{G}\neq 1$. We show next that for every $p$-element $1\neq \overline{x}=x\ce{G}{\rad{p}{G}}\in\overline{A}\cup \overline{B}$ vanishing in $\overline{G}$ it holds that $i_{\overline{G}}(\overline{x})$ is a $p$-number, and then (1) applies. Since by Remark \ref{remarkhyp} we can suppose that $x\in A\cup B$ is a $p$-element vanishing in $G$, by assumptions we get that $i_G(x)$ is a prime power (actually a $p$-number, because $x\notin \ce{G}{\rad{p}{G}}$). Therefore $i_{\overline{G}}(\overline{x})$ is also a $p$-number and we are done.

(3) Let denote $\overline{G}:=G/\fit{G}$, and let assume $\overline{G}\neq 1$. If the statement is false, then by Theorem \ref{SylowNormal} there exists a vanishing $p$-element $1\neq \overline{x}=x\fit{G}$ in $\overline{A}\cup \overline{B}$. By Remark \ref{remarkhyp}, $x\notin \fit{G}$ is a vanishing $p$-element in $A\cup B$, and so $i_G(x)$ is a power of a prime $q\neq p$. It follows $x\in \op{F}_2(G)$ by Proposition \ref{CaminaCamina}, so $1\neq \overline{x}\in \rad{p}{\overline{G}}$. Proposition \ref{lemma_non_brough} implies that $p$ divides $i_{\overline{G}}(\overline{x})$, and so $p$ divides $i_G(x)$, the final contradiction.

(4) We proceed by induction on $\abs{G}$ in order to show that $P\rad{p'}{G}$ is normal in $G$. We may assume $\rad{p'}{G}=1$, and by (3) we get that $P\fit{G}=P$ is normal in $G$. The second assertion about the $p$-solubility of $G$ follows directly.
\end{proof}

\medskip

We remark that the vanishing analogue of \cite[Theorem B (2)]{FMOprime} is not true, that is, if the considered vanishing indices are powers of primes distinct from $p$, then the Sylow $p$-subgroup might not be abelian:

\begin{example}
Let $G$ be a Suzuki group of degree 8, and let $H$ be the normaliser of a Sylow $2$-subgroup of $G$. Then $H$ is a core-factorisation of its Sylow subgroup of order $2$ and a Sylow subgroup of order $7$, and $H$ does not have vanishing $2$-elements. Nevertheless, the Sylow $2$-subgroup of $H$ is non-abelian.
\end{example}

Moreover, \cite[Theorem B]{FMOprime} asserts that if all the $p$-elements in a factor have prime power indices in the whole factorised group, then there is a unique prime that divides all the considered indices. However, we do not know if the vanishing version of this fact is true.

Finally, note that if we consider the assumptions in Theorem \ref{theoremCCvanishing} for every prime in $\pi(G)$, then the third statement tells us that $G/\fit{G}$ is nilpotent. In fact, the following result shows that $G/\fit{G}$ is abelian for such a group (compare with \cite[Corollary C (1)]{FMOprime}).

\begin{corollary}
\label{G/Fabelian}
Let $G=AB$ be a core-factorisation. If every prime power order element $x\in A\cup B$ vanishing in $G$ has prime power index, then $G/\fit{G}$ is abelian. In particular, if these prime powers are actually  $p$-numbers for a prime $p$, then $G$ has a normal Sylow $p$-subgroup and abelian Hall $p'$-subgroups.
\end{corollary}

\begin{proof}
$G/\fit{G}$ is nilpotent by Theorem \ref{theoremCCvanishing} (3). Let denote by $\overline{G}:=G/\fit{G}$, and let assume that $\overline{G}\neq 1$ and that there exists $1\neq \overline{x}=x\fit{G}$ a prime power order element in $\overline{A}\cup \overline{B}$ vanishing in $\overline{G}$. Then $\overline{x}$ is a $p$-element for some prime $p$, and we may suppose $x\in (A\cup B)\smallsetminus\fit{G}$ is a $p$-element vanishing in $G$. By assumption, we have that $i_G(x)$ is a prime power. Since $\overline{G}$ is nilpotent, then by Proposition \ref{lemma_non_brough} it follows that $i_{\overline{G}}(\overline{x})$ is a $p$-number, and so is $i_G(x)$. It follows by Wielandt's lemma that $x\in\rad{p}{G}$, so $\overline{x}=1$, a contradiction. Thus $\overline{G}$ does not have any vanishing prime power order element in $\overline{A}\cup \overline{B}$, and by Corollary \ref{propositionKEY} we get that it is abelian.

For the second assertion, note that $P$ is the unique Sylow $p$-subgroup of $G$ by Theorem \ref{theoremCCvanishing} (1), so we claim that $H\cong G/P$ is an abelian Hall $p'$-subgroup of $G$. Let denote $\tilde{G}:=G/P$, so $\tilde{G}=\tilde{A} \tilde{B}$. Hence, $\tilde{G}$ does not have any vanishing prime power order element in $\tilde{A}\cup\tilde{B}$, since otherwise those elements are central by our assumptions, a contradiction. So it follows by Corollary \ref{propositionKEY} that $\tilde{G} = G/P\cong H$ is abelian.
\end{proof}


\section{Square-free vanishing indices}
\label{sec_square}

In this last section we focus on vanishing indices in factorised groups which are square-free, motivated by previous developments in \cite{B, BK, DPS}. The next theorem treats the most extreme square-free case: when the vanishing indices are not divisible by a fixed prime $p$. We should comment that, although some arguments in the proof of the first statement are similar to those in \cite[Theorem 3.3]{BK}, we include them here for the sake of comprehensiveness.

\begin{theorem}
\label{theorem_p-reg}
Let $G=AB$ be a core-factorisation.
\begin{enumerate}
	\item[\emph{(1)}] Assume that $p$ does not divide $i_G(x)$ for every $p$-regular element of prime power order $x\in A\cup B$ vanishing in $G$. Then $G$ is $p$-nilpotent.
	
	\item[\emph{(2)}] If $p$ does not divide $i_G(x)$ for every prime power order element $x\in A\cup B$ vanishing in $G$, then $G$ is $p$-nilpotent with abelian Sylow $p$-subgroups.
\end{enumerate}
\end{theorem}

\begin{proof}
(1) Assume the result is false. We argue with $G$
a minimal counterexample to the theorem. By minimality, we may suppose that $\rad{p'}{G}=1$. Let $N$ be a minimal normal subgroup of $G$ such that $N\leqslant A$, for instance. If $N$ is soluble, since $p$ divides its order it follows that $N$ is a $p$-group. We can assume that $N$ is proper in $G$ since otherwise $G$ is a $p$-group, so by minimality we get
that $G/N$ is $p$-nilpotent. Hence $G$ is $p$-separable, and $\ce{G}{\rad{p}{G}}\leqslant\rad{p}{G}$. This
last fact and our assumptions produce that there are no $p$-regular elements of prime power order $x\in A\cup B$ vanishing in $G$, and Corollary \ref{corollaryHall} applies with $\sigma=p'$. Thus $N$ is non-soluble, and applying the same arguments as in the second paragraph in the proof of \cite[Theorem 3.3]{BK}, it can be obtained a $p$-regular element of prime power order in $N\leqslant A$ which is vanishing in $G$ and whose conjugacy class size in $G$ is divisible by $p$, the final contradiction.

(2) $G$ is $p$-nilpotent by (1). Let denote $\tilde{G}:=G/H$ where $H$ is the unique Hall $p'$-subgroup of $G$, and then $\tilde{G}=\tilde{A} \tilde{B}$. Then, $\tilde{G}$ does not have any vanishing prime power order element in $\tilde{A}\cup\tilde{B}$, because otherwise the hypotheses imply that those elements are central, a contradiction. Now in virtue of Corollary \ref{propositionKEY} we get that $\tilde{G}$ is abelian.
\end{proof}

\medskip

Note that Theorem \ref{theorem_p-reg} provides a vanishing version of \cite[Theorem 1.1]{BCL} for products of two groups, even relaxing the mutual permutability of the factors. We also remark that \cite[Theorem 3.3]{BK} is Theorem \ref{theorem_p-reg} (1) for the trivial factorisation. Indeed, (2) implies the next corollary, which improves the main result of \cite{DPS} by considering only vanishing indices of prime power order elements:

\begin{corollary}
Let $G$ be a group, and $p$ be a prime. If $p$ does not divide any vanishing index of a prime power order element, then $G$ is $p$-nilpotent with abelian Sylow $p$-subgroups.
\end{corollary}

Regarding square-free vanishing indices, we first analyse those which are not divisible by $p^2$, for a fixed prime $p$. The next proposition is actually the vanishing version of \cite[Theorem A]{FMOsquare}. We point out that this result is valid for any arbitrary factorisation of a $p$-group.

\begin{proposition}
Let $p$ be a prime number and let $P=AB$ be a $p$-group such that $p^2$ does not divide $i_P(x)$ for all $x\in A\cup B$ vanishing in $P$. Then $P' \leqslant \fra{P} \leqslant \ze{P}$, $P'$ is elementary abelian and $\abs{P'}\leq p^2$.
\label{knoche_prop}
\end{proposition}

\begin{proof}
Since the non-vanishing elements of a $p$-group lie in its centre because of Proposition \ref{INW_supersoluble}, we can apply directly \cite[Theorem A]{FMOsquare} in order to get the thesis.
\end{proof}

\medskip

The following lemma will be essential in the sequel.

\begin{lemma} \emph{\cite[Lemma 2.4]{BCL}}
\label{lemabcl}
Let $p$ be a prime, and $Q$ be a $p'$-group acting faithfully on an elementary abelian $p$-group $N$ with $\abs{[x, N]} = p$, for all $1\neq x\in Q$. Then $Q$ is cyclic.
\end{lemma}

In \cite{B}, the author posed the following question: a group such that all its vanishing indices are not divisible by $p^2$, for a prime satisfying $(p-1, \abs{G})=1$, must be $p$-nilpotent? The following theorem gives a positive answer to this question, even for some factorised groups (see Corollary \ref{corollary_pnilp} for the case $G=A=B$).

\begin{theorem}
\label{theoremp2}
Let $G=AB$ be a core-factorisation, and let $p$ be a prime such that $(p-1, \abs{G})=1$. Suppose that $i_G(x)$ is not divisible by $p^2$ for every prime power order element $x\in A\cup B$ vanishing in $G$. It follows that:
\begin{enumerate}
	\item[\emph{(1)}] $G$ is soluble.
	
	\item[\emph{(2)}] $G$ is $p$-nilpotent.
	
	\item[\emph{(3)}] If $P\in\syl{p}{G}$, then $P'\leqslant \fra{P}\leqslant\ze{P}$, $P'$ is elementary abelian and $\abs{P'}\leq p^2$.
\end{enumerate} 
\end{theorem}

\begin{proof}
(1) Suppose that the result is false and let $G$ be a counterexample of minimal order. Since every group of odd order is soluble, we may assume that $p=2$ because $(p-1, \abs{G})=1$. The class of soluble groups is a saturated formation, so we can suppose that there exists a unique minimal normal subgroup $N$. Moreover, $N$ is non-soluble. We have for instance $N\leqslant A$, because $G=AB$ is a core-factorisation. Then it is enough to reproduce the arguments in the proof of \cite[Theorem 3.1]{BK} to obtain a prime power order element in $N\leqslant A$ which is vanishing in $G$ and whose conjugacy class size is divisible by $4$, a contradiction.

(2) Assume that the result is not true and let $G$ be a counterexample of least possible order. By the minimality of $G$ we may suppose that $\rad{p'}{G}=1$. Let $N$ be a minimal normal subgroup of $G$. Thus $p$ divides its order and, since $G$ is soluble by (1), then $N$ is $p$-elementary abelian. Moreover, the class of $p$-nilpotent groups is a saturated formation, so $N$ is the unique minimal normal subgroup of $G$ and by \cite[A - 15.6, 15.8]{DH} we get $N=\rad{p}{G}=\fit{G}=\ce{G}{N}$. We can consider $N\leqslant A$, for instance. We take $K/N$ a minimal normal subgroup of $G/N$ such that it is covered by either $A$ or $B$. We claim that each element in $K\smallsetminus N$ is vanishing in $G$. Since $N=\rad{p}{G}$, then $K/N$ is $q$-elementary abelian for some prime $q\neq p$. Indeed, we get $\ce{K}{N}\leqslant\ce{G}{N}=N$. It follows by Lemma \ref{contradiction_lemma} that every element in $K\smallsetminus N$ is vanishing in $G$.

Note that $K=[N]Q$ where $Q\in\syl{q}{K}$ is elementary abelian. If we take $1\neq xN\in K/N$, then we can assume that $x\in K\smallsetminus N$ is a $q$-element in $A\cup B$ by conjugation. Hence $p^2$ does not divide $i_G(x)=\abs{G:\ce{G}{x}}$. Note that the $p$-number $1\neq \abs{N:\ce{N}{x}}$ divides $i_G(x)$. On the other hand, $x$ acts coprimely on $N$, which is abelian, so $N=\ce{N}{x} \times [N, x]$. It follows $\abs{[N, x]}=p$. Observe that $\ce{Q}{N}=Q\cap \ce{G}{N}=1$, so $Q$ acts faithfully and coprimely on $N$. Further, if $1\neq y\in Q$, then $y\in K\smallsetminus N$ and by the previous argument we get $\abs{[N, y]}=p$. Now Lemma \ref{lemabcl} leads to the fact that $Q$ is cyclic, so $\abs{K/N}=q$ and $K=N\langle x\rangle$. Hence $\ce{N}{x}=\ce{N}{K}$ is normal in $G$. Since $\ce{N}{x}<N$, by the minimality of $N$ we obtain $\ce{N}{x}=1$ and so $N=[N, x]$ has order $p$. Now $G/N=\no{G}{N}/\ce{G}{N}$ is isomorphich to a subgroup of $\op{Aut}(N)$, which is isomorphic to $C_{p-1}$. It follows that $\abs{G/N}$ divides both $p-1$ and $\abs{G}$, the final contradiction.

(3) Notice that $P\in\syl{p}{G}$ is isomorphic to $G/\rad{p'}{G}$ by the previous assertion. Hence the result follows by Proposition \ref{knoche_prop}.
\end{proof}

\begin{corollary}
\label{corollary_pnilp}
Let $G$ be a group, and let $p$ be a prime such that $(p-1, \abs{G})=1$. Assume that $p^2$ does not divide $i_G(x)$ for each prime power order element $x$ vanishing in $G$. Then $G$ is a soluble $p$-nilpotent group. Moreover, if $P\in \syl{p}{G}$, then $P'\leqslant \fra{P}\leqslant\ze{P}$, $P'$ is elementary abelian and $\abs{P'}\leq p^2$.
\end{corollary}

In \cite[Theorem B (c)]{FMOsquare} it is proved the following: ``Let $G=AB$ be the product of the mutually permutable subgroups $A$ and $B$. Let $p$ be a fixed prime satisfying $(p-1, \abs{G})=1$. If all $p$-regular prime power order elements in $A\cup B$ have $i_G(x)$ not divisible by $p^2$, then $G/\rad{p}{G}$ has elementary abelian Sylow $p$-subroups''. We point out that this property does not remain true under the hypotheses of Theorem \ref{theoremp2}, as the following example shows:

\begin{example}
Let $G=[A]B$ be the semidirect product of a cyclic group $B$ of order $4$ which acts transitively on a cyclic group $A$ of order $5$. Let the prime $p=2$. Then $G=AB$ is a core-factorisation, and all the vanishing elements of $G$ (not only those lying in $A\cup B$) have index not divisible by $4$. However, $\rad{2}{G}=1$ and $G/\rad{2}{G}$ does not have elementary abelian Sylow $2$-subgroups.
\end{example}

We highlight that the arguments used in \cite[Theorem C]{FMOsquare} can be generalised in order to obtain the following more general result for core-factorisations.

\begin{theorem}
\label{p-sol_vanishing}
Let $G=AB$ be a core-factorisation, and let $p$ be a prime. Suppose that for every prime power order $p$-regular element $x\in A\cup B$ vanishing in $G$, $i_G(x)$ is not divisible by $p^2$. If $G$ is $p$-soluble, then $G$ is $p$-supersoluble.
\end{theorem}

\begin{proof}
It is sufficient to follow the proof of \cite[Theorem C]{FMOsquare}. Notice that, in this case, we can use Lemma \ref{contradiction_lemma2} with a minimal normal subgroup $Z/N$ of $G/N$ such that it lies in either $AN/N$ or $BN/N$. Thus, we can affirm that every element in $Z\smallsetminus N$ is vanishing in $G$, in order to apply the assumption that $i_G(x)$ is not divisible by $p^2$ for every prime power order $p$-regular element $x\in Z\smallsetminus N$. 
\end{proof}

\medskip

When we consider square-free indices for all primes in $\pi(G)$, we get:

\begin{theorem}
\label{theoremSquare}
Let $G=AB$ be a core-factorisation. Suppose that $i_G(x)$ is square-free for every prime power order element $x\in A\cup B$ vanishing in $G$. Then:
\begin{enumerate}
	\item[\emph{(1)}] $G$ is supersoluble.
	
	\item[\emph{(2)}] $G'$ is abelian.
	
	\item[\emph{(3)}] $G'$ has elementary abelian Sylow subgroups.
	
	\item[\emph{(4)}] $\fit{G}'$ has Sylow $p$-subgroups of order at most $p^2$, for each prime $p$.
\end{enumerate} 
\end{theorem}

The statements (2-3-4) above are immediate consequences of the next more general result for an arbitrary factorisation of a supersoluble group.

\begin{theorem}
Let $G=AB$ be the product of the subgroups $A$ and $B$, and assume that $G$ is supersoluble. Suppose that $i_G(x)$ is square-free for every prime power order element $x\in A\cup B$ vanishing in $G$. Then:
\begin{enumerate}
	\item[\emph{(1)}] $G'$ is abelian.
	
	\item[\emph{(2)}] $G'$ has elementary abelian Sylow subgroups.
	
	\item[\emph{(3)}] $\fit{G}'$ has Sylow $p$-subgroups of order at most $p^2$, for each prime $p$.
\end{enumerate} 
\end{theorem}

\begin{proof}
We adapt the proof of \cite[Theorem D]{FMOsquare} for our hypotheses regarding vanishing elements.

To prove either (1) or (2), arguing by minimal counterexample in each case we can assume that there exists a prime $p$ such that $\fit{G}=\rad{p}{G}=P$ is a Sylow $p$-subgroup of $G$. Since $G$ is supersoluble, then Proposition \ref{INW_supersoluble} yields that every $q$-element ($q\neq p$) $x\in A\cup B$ is vanishing in $G$. Thus we can apply for such an element the class size hypothesis as in \cite[Theorem D (1-2)]{FMOsquare}.

On the other hand, following the proof of \cite[Theorem D (1)]{FMOsquare}, we need to assure that if $N$ is a prefactorised normal subgroup of $G$ which contains $P$ as a Sylow $p$-subgroup, then $N$ inherits the assumptions. Let see that this fact also holds in our case. First, such an $N$ is clearly supersoluble and $N=PH=(N\cap A)(N\cap B)$ where $H\in\hall{p'}{N}$. Hence, if we take a $p$-element $x\in (N\cap A)\cup(B\cap B)$ vanishing in $N$, then it follows that $x$ is vanishing in $G$; otherwise we get by Proposition \ref{INW_supersoluble} that $x\in\ze{P}$ because $\fit{G}=\rad{p}{G}=P\neq 1$, and so Proposition \ref{lemma_non_brough} leads to the fact that $x$ is non-vanishing in $N$, a contradiction. Moreover, all the $q$-elements in $(N\cap A)\cup(N\cap B)$ are vanishing in $G$ by the above paragraph. Thus, in both cases, $i_N(x)$ divides $i_G(x)$ which is square-free. Note that, in particular, if $H=1$, then $N=P$ satisfies the hypotheses of Proposition \ref{knoche_prop}.

The statement (3) runs as in the proof \cite[Theorem D (3)]{FMOsquare}, but applying Proposition \ref{knoche_prop}, which is the vanishing version of \cite[Theorem A]{FMOsquare}.
\end{proof}

\bigskip

\begin{proof}[Proof of Theorem \ref{theoremSquare}]
(1) Considering the smallest prime divisor of $\abs{G}$ and Theorem \ref{theoremp2} (1), we conclude that $G$ is soluble. Hence, it is $p$-soluble for each prime $p$. 
Applying Theorem \ref{p-sol_vanishing}, we get that $G$ is $p$-supersoluble for each prime $p$, so it is supersoluble.

(2-4) These assertions follow from the previous theorem.
\end{proof}

\medskip

In \cite[Theorem 3.2]{BK} the author gives a supersolubility criterion for a group such that every vanishing index of a prime power order element is square-free. We want to highlight that the following consequence of Theorem \ref{theoremSquare} gives more information on the structure of such a group. Moreover, our techniques differ from those used in \cite[Theorem 3.2]{BK}.

\begin{corollary}
Let $G$ be a group, and let assume that $i_G(x)$ is square-free for each prime power order element $x$ vanishing in $G$. Then $G$ is supersoluble, and $G'$ is abelian with elementary abelian Sylow subgroups. Further, $\fit{G}'$ has Sylow $p$-subgroups of order at most $p^2$, for each prime $p$.
\end{corollary}



\end{document}